\documentclass[10pt, a4paper, reqno, english]{amsart}
\usepackage{color}

\usepackage[utf8]{inputenc}
\usepackage{hyperref}
\usepackage{amssymb}
\usepackage{enumitem}

\newcommand\NN{\mathbb{N}}

\newcommand\RR{\mathbb{R}}
\newcommand\QQ{\mathbb{Q}}

\newcommand\PP{\mathbb{P}}

\newcommand{\OO}{\mathcal{O}}
\newcommand{\K}{K}

\newcommand{\Ok}{\OO_K}
\newcommand{\units}{{\Ok^\times}}

\newcommand{\id}[1]{\mathfrak{#1}}
\newcommand{\ppp}{\id p}
\newcommand{\aaa}{\id a}
\newcommand{\bbb}{\id b}

\newcommand{\ddd}{\id d}

\newcommand{\places}{{\Omega_K}}
\newcommand{\archplaces}{{\Omega_\infty}}
\newcommand{\nonarchplaces}{{\Omega_0}}
\newcommand{\abs}[1]{\left|#1\right|}
\newcommand{\absv}[1]{\left|#1\right|_v}
\newcommand{\norm}[1]{\left\lVert #1 \right\rVert}

\DeclareMathOperator{\N}{\mathfrak{N}}

\newcommand\xx{\mathbf{x}}

\newcommand\where{\,:\,}
\newcommand\vz{\mathbf{0}}
\newcommand\vl{{\boldsymbol{\lambda}}}

\DeclareMathOperator{\vol}{vol}

\DeclareMathOperator\Res{Res}

\newcommand\vq{\mathbf{q}}
\newcommand\locdegv{{n_v}}
\newcommand\dg{n}
\newcommand\qfheight[1]{{\langle #1 \rangle}}

\newtheorem{theorem}{Theorem}
\newtheorem{lemma}[theorem]{Lemma}

\theoremstyle{definition}

\newtheorem*{acknowledgements}{Acknowledgements}

\numberwithin{theorem}{section}
\numberwithin{equation}{section}

\begin{document}

\setcounter{tocdepth}{1}

\title[
Rational points on smooth cubic surfaces
]{Counting rational points on smooth cubic surfaces 
}

\address{Institut für Analysis und Zahlentheorie, Technische Universität Graz,
 Kopernikusgasse 24/II, 8010 Graz, Austria}
\author{Christopher Frei} \email{frei@math.tugraz.at}

\address{Mathematisch Instituut Leiden, Universiteit Leiden,
2333 CA Leiden, Netherlands}
\author{Efthymios Sofos} \email{e.sofos@math.leidenuniv.nl}

\date{September 23, 2014}

\begin{abstract}
  We prove that any smooth cubic surface defined over any number field
  satisfies the lower bound predicted by Manin's conjecture possibly
  after an extension of small degree.
\end{abstract}


\maketitle

\section{Introduction}
Let $\K$ be a number field.  Assume $X\subset \PP^3_\K$ is a smooth cubic
surface defined over $\K$ for which the set of rational points $X(\K)$
is not empty.  We are concerned with estimating the number of
rational points of bounded height on $X$.  Let $U\subset X$ be
the Zariski-open set obtained by removing the lines contained in $X$,
denote by $H$ the exponential Weil height on $\PP^3(\K)$ and define
for all $B\geq 1$ the counting function
\begin{equation*}
  N_{K,H}(U,B):= \sharp\{\xx \in U(\K) \where H(\xx)\leq B\}.
\end{equation*}
Manin's conjecture \cite{manin}
for smooth cubic surfaces states that
\begin{equation}
\label{eq:manin}
N_{K,H}(U,B) \sim c B (\log B)^{\rho_{X\!,\K} -1},
\end{equation}
as $B \to \infty$, where $\rho_{X\!,\K}$ denotes the rank of the Picard
group of $X$ over $K$ and $c=c_{\K,H,X}$ is a positive constant which
was later interpreted by Peyre \cite{peyre}.

There has been a wealth of results towards this conjecture but it has
never been established for a single smooth cubic surface over any
number field. There are proofs of Manin's conjecture for certain
singular cubic surfaces over $\QQ$, e.g.~\cite{MR2332351}, and other
number fields \cite{MR1679843, MR3107569, E6}, but here we will only
consider the smooth case. Heath-Brown
\cite{FourThirds}, building upon the work of Wooley \cite{trev},
proved, using a fibration argument, that if $X$ is a smooth cubic
surface defined over $\QQ$ that contains $3$ rational coplanar lines then $N_{\QQ,H}(U,B) \!  \ll_{X,\epsilon} \!
B^{\frac{4}{3}+\epsilon}$ holds for any $\epsilon>0$. This result was
subsequently extended to arbitrary number fields by Broberg
\cite{broberg} and Browning and Swarbrick Jones \cite{mikew}.
Heath-Brown \cite{mestre} revisited the subject by proving that a
bound of the same order holds for all smooth cubic surfaces defined
over $\QQ$ subject to a standard conjecture regarding the growth rate
of the rank of elliptic curves.  Using a generalization of
Heath-Brown's determinant method, Salberger \cite{salberger} was able
to prove unconditionally that one has
$N_{\QQ,H}(U,B)\!\ll_{\epsilon}\!  B^{\frac{12}{7}+\epsilon}$ for
arbitrary smooth cubic surfaces defined over $\QQ$ and for all
$\epsilon>0$.
Regarding lower bounds, the only available result is
due to Slater and Swinnerton-Dyer \cite{MR1679836} who used a secant
and tangent process to establish that $N_{\QQ,H}(U,B)\!\gg_{X}\!  B
(\log B)^{\rho_{X\!,\QQ} -1}$ whenever $X$ 
has
$2$ skew
lines defined over $\QQ$.

Our main result shows that for all smooth cubic surfaces over any
number field $L$, the lower bound predicted by Manin's conjecture has
the correct order of magnitude as soon as one passes to a sufficiently
large extension of $L$. Some context for this type of result is
provided by the formulation of Manin's conjecture in \cite{MR1032922} and
by the notion of potential density \cite[\S 3]{MR2275615}.

\begin{theorem}
\label{thm:all}
  Let
$X$ be any smooth cubic surface
defined
over any number field
 $L$.
Then there exists an extension
$\K_0\!$ of $L$ with $[\K_0\!:\!L] \leq 432$
such that
for all number fields
$\K \supseteq \K_0$
we have
  \begin{equation*}
    N_{K,H}(U,B) \gg B(\log B)^{\rho_{X\!,\K}-1},
  \end{equation*}
as $B\to \infty$, where the implicit constant
depends at most on $X$ and $\K$.
\end{theorem}

We hope that the number $432$ will serve as a useful benchmark for researchers
in the area to compare the strength of other methods with in the future.

Our Theorem \ref{thm:all} is a consequence of Theorem \ref{thm:ssd}
below, which furthermore provides an explicit description of
$\K_0$. One can take $\K_0$ to be any extension of $L$ over which $2$
skew lines of $X$ are defined. The fact that there exists such a
$\K_0$ with $[\K_0:L]\leq 432=27 \cdot 16$ can be proved as follows.
Since $X$ contains exactly $27$ lines, each of them is defined over an
extension of degree at most $27$. Since there are $16$ complex lines
skew to a line $\ell$, a further extension of degree at most $16$
ensures that a line skew to $\ell$ is defined.

\begin{theorem}
\label{thm:ssd}
  Let
$X$ be a smooth cubic surface
defined
over any number field
 $\K$ such that $X$
contains
two skew lines defined over $\K$.
Then 
  \begin{equation*}
    N_{\K,H}(U,B) \gg B(\log B)^{\rho_{X\!,\K}-1},
  \end{equation*}
as $B\to \infty$, where the implicit constant
depends only on $\K$ and $X$.
\end{theorem}
Theorem \ref{thm:ssd} is a generalization of Slater and
Swinnerton-Dyer's result to arbitrary number fields. 
Our proof however
is entirely
different
and more conceptual
than the one 
of Slater and
Swinnerton-Dyer.  
It relies on a conic
bundle
fibration of $X$ and a number field version of the earlier work
\cite{sofos} of the second author which allows us to count rational
points on each conic individually.

This result is presented in Section \ref{sec:preliminaries}, together
with our main analytic tool, a variant of Wirsing's theorem. Theorem
\ref{thm:ssd} will be proved in Sections
\ref{sec:conic_bundle_structure} and \ref{sec:proof}.

Throughout this paper, all implied constants are allowed to depend on
the cubic surface $X$ and the underlying number field $\K$, unless the
contrary is explicitly stated.

\section{Preliminaries}
\label{sec:preliminaries}
We denote the degree of $\K$ by $n$, its discriminant by $\Delta_K$,
and its ring of integers by $\Ok$. We write $\archplaces$,
$\nonarchplaces$ and $\places$ for the sets of archimedean places,
non-archimedean places, and all places of $\K$, respectively. We will
write $h_\K$, $R_\K$ and $\mu_\K$ for the class number, regulator and
the group of roots of unity in $\K$. Moreover, $r_1$ (resp. $r_2$)
denotes the number of real (resp. complex) embeddings of $\K$.
 
In the proof of Theorem \ref{thm:ssd}, we 
only 
need to consider a
special family of height functions on $\PP^2(K)$. Let $\vl =
(\lambda_v)_{v\in\archplaces}\in (0,\infty)^\archplaces$. For every $v
\in \places$, and $\xx=(x,y,z)\in K_v^3$, let
\begin{equation}\label{eq:def_norms}
  \norm{\xx}_{\vl, v} :=
  \begin{cases}
    \max\{\absv{x}, \lambda_v\absv{y},\absv{z}\} &\text{ if }v\in\archplaces\\
    \max\{\absv{x}, \absv{y},\absv{z}\} &\text{ if }v\in\nonarchplaces.
  \end{cases}
\end{equation}
Here, $\absv{\cdot}$ is the unique absolute value on $\K_v$ extending
the usual absolute value on $\QQ_p$, if $v$ lies over the place $p$ of $\QQ$. Let $\locdegv := [\K_v :\QQ_p]$. We consider heights on $\PP^2(\K)$ defined by
\begin{equation*}
  H_{\vl}((x:y:z)) := \prod_{v\in\places}\norm{(x,y,z)}_{\vl, v}^{\locdegv}.
\end{equation*}

Let $C \subset \PP^2_\K$ be a nonsingular conic defined by a quadratic
form $Q \in \Ok[x,y,z]$ and assume that $C(\K)\neq \emptyset$, which
implies that $C \cong \PP^1_\K$.  The heights $H_\vl$ induce
heights on $C(\K)$ via the embedding $C \subset \PP^2_\K$.  We are
interested in estimating the quantity
\begin{equation*}
  N_{K,H_\vl}(C,B) := \sharp\left\{\xx \in C(\K)\where H_{\vl}(\xx)\leq B\right\}
\end{equation*}
when the underlying quadratic form has the special shape
\begin{equation}\label{eq:Q_special_shape}
  Q = ax^2 + bxy + dxz + eyz + fz^2,
\end{equation}
with $a,b,d,e,f \in \Ok$. It is a simple task to write down an
explicit isomorphism between $C$ and $\PP^1_K$. Let $\Pi$ be the matrix
\begin{equation*}
  \Pi := 
  \begin{pmatrix}
    b&e&0\\-a&-d&-f\\0&b&e
  \end{pmatrix},
\end{equation*}
and define
\begin{equation}\label{eq:def_q}
  \vq(u,v) := \Pi \cdot
  \begin{pmatrix}
    u^2\\uv\\v^2
  \end{pmatrix}.
\end{equation}
Then the map $(u,v)\mapsto \vq(u,v)$ induces an isomorphism $\PP^1_K
\to C$. To measure the form $Q$ and the height $H_\vl$, we introduce quantities
\begin{equation*}
  \qfheight{Q} := \prod_{v\in\places}\max\{\absv{a},\absv{b},\absv{d},\absv{e},\absv{f}\}^{n_v} \text{ and }\   M_\vl := \prod_{v\in\archplaces}\max\{1, \lambda_v^{-1}\}^{n_v}.
\end{equation*}
The following lemma is a number field version of
\cite[Prop. 2.1]{sofos}, specialized to the heights $H_\vl$ and with a
crude estimation of the error term. 

\begin{lemma}\label{prop:conics}
There exist constants
$\beta \in (0,1/2)$
and
$\gamma>0$
which depend at most on $\K$
such that
whenever $C \subset \PP^2_K$ is a nonsingular conic defined by a
  quadratic form $Q$ as in \eqref{eq:Q_special_shape}, and $\vl \in
  (0,\infty)^\archplaces$,
  then
  \begin{equation*}
    N_{\K, H_\vl}(C,B) = c_{\K, \vl, C}\cdot B + O\left(B^{1-\beta} (M_\vl \qfheight{Q})^{\gamma}\right),
  \end{equation*}
  for $B \geq 1$. The
  leading constant $c_{\K,\vl,C}$ is positive and is the one predicted by Peyre,
  and the implied constant in the error term depends only on $\K$.
\end{lemma}

Of course, Manin's conjecture for conics with respect to arbitrary
anticanonical height functions is already known \cite{peyre}, so the
novelty of Lemma \ref{prop:conics} lies in the uniformity of the
estimate in the coefficients of the underlying quadratic form.

The proof over $\QQ$ in \cite{sofos} is based on the parameterization
of $C(K)$ via $\vq$, which reduces the estimation of
$N_{\K,H_\vl}(C,B)$ to a lattice point counting argument. The same
reduction works over arbitrary number fields by considering primitive
points with respect to a fixed set of representatives for the ideal
classes of $\Ok$ and suitably chosen fundamental domains for the
action of the unit group. The resulting lattice point counting problem
can then be solved using, for example, the main result from
\cite{arXiv:1210.5943}. The special shape of the heights $H_\vl$
enters only here, to ensure definability in an o-minimal
structure. Altogether, the passage from $\QQ$ to arbitrary number
fields in the proof of Lemma \ref{prop:conics} uses mostly arguments
already given in \cite{Marta}, but is straightforward and much
simpler. The proof provides explicit values $\beta=1/(3n)$ and
$\gamma=4$, but we will not give further details here. For the purpose
of proving Theorems \ref{thm:all} and \ref{thm:ssd} we do not need
explicit values for $\beta$ and $\gamma$ since any polynomial saving
in terms of $B$ and any polynomial dependence on $\qfheight{Q}$ and
$M_\vl$ in the error term suffices.

As usual, the constant $c_{\K,\vl,C}$ has an explicit expression of
the form
 \begin{equation}\label{eq:constant_product_formula}
   c_{\K,\vl,C} = \frac{1}{2}\cdot \frac{2^{r_1}(2\pi)^{r_2}h_K R_K}{|\mu_K|}\cdot\frac{1}{\abs{\Delta_K}}\cdot\prod_{v\in\places}\sigma_v,
 \end{equation}
 with local densities $\sigma_v$ given as follows. For
 $v\in\archplaces$, we have
\begin{equation}\label{eq:sigma_v_arch}
  \sigma_v = \vol\{(y_1,y_2)\in \K_v^2 \where \norm{\vq(y_1,y_2)}_{\vl, v} \leq 1\}\cdot
  \begin{cases}
    1 &\text{ if $v$ is real,}\\
    4/\pi &\text{ if $v$ is complex,}
  \end{cases}
\end{equation}
where $\vol(\cdot)$ denotes the usual Lebesgue measure on $\K_v^2 \cong
\RR^{2\locdegv}$.  For $v \in \nonarchplaces$ corresponding to a prime
ideal $\ppp$ of $\Ok$, we have
\begin{equation}\label{eq:sigma_v_nonarch}
  \sigma_v = 1-\frac{1}{\N\ppp^2} + \left(1-\frac{1}{\N\ppp}\right)\sum_{d\in\NN}\frac{\rho_\vq^*(\ppp^d)}{\N\ppp^d},
\end{equation}
where, for any ideal $\aaa$ of $\Ok$,
the function
$  \rho_\vq^*(\aaa)$
is defined as
\begin{equation}\label{eq:def_rho}
 \sharp\{(\sigma,\tau)\in (\Ok/\aaa)^2 \where \sigma\Ok+\tau\Ok+\aaa=\Ok,\ \vq(\sigma,\tau)\equiv \vz \bmod\aaa\}.
\end{equation}

The following version of Wirsing's theorem is a straightforward generalization to number fields of \cite[Theorem A.5]{OperaDeCribro}. Its proof is, mutatis mutandis, the same and therefore omitted.

\begin{lemma}\label{thm:wirsing}
  Let $g$ be a multiplicative function on nonzero ideals of $\Ok$ that
  is supported on the set of squarefree ideals. Assume that
we have
  \begin{equation}\label{eq:a15}
    \sum_{\N\ppp\leq x}g(\ppp)\log(\N\ppp) = k \log x + O(1)
  \end{equation}
  for all $x \geq 2$, with $k \geq -1/2$, where the sum runs over nonzero
  prime ideals $\ppp$ and the implied constant is allowed to depend at
  most on $K$ and $g$.  Assume, moreover, that
  \begin{equation}\label{eq:a16}
    \prod_{w\leq \N\ppp < z}\left(1 + \abs{g(\ppp)}\right)\ll \left(\frac{\log z}{\log w}\right)^{\abs{k}}
  \end{equation}
holds for all $z > w \geq 2$ and that
\begin{equation}\label{eq:a17}
  \sum_{\ppp}g(\ppp)^2\log(\N\ppp) < \infty.
\end{equation}
Then
\begin{equation*}
  \sum_{\N\aaa\leq x}g(\aaa) = c_g (\log x)^k + O((\log x)^{|k|-1}),
\end{equation*}
with a positive constant $c_g,$
where the implied constant depends at most on $\K$ and $g$.
\end{lemma}

\section{Covering the cubic surface with conics}\label{sec:conic_bundle_structure}
Let $\K$ be a number field and $X \subset \PP^3_\K$ a smooth cubic
surface containing two skew lines defined over $\K$.  The residual
intersection of $X$ with a plane containing the first line generically
defines a smooth conic. The second line contained in $X$ intersects
each such plane in a point that necessarily lies in the residual conic,
thus showing that it is isotropic over $\K$.

The construction we have described does in fact yield a conic bundle morphism.
A linear change of variables
allows us to assume that the two skew $\K$-lines
are given by
\begin{equation*}
  x_0=x_1=0\quad \text{ and }\quad x_2=x_3=0,
\end{equation*}
whence
the cubic form defining $X$ has the shape
\begin{equation}\label{eq:cubic_form}
  F = a(x_0,x_1)x_2^2 + d(x_0,x_1)x_2x_3 + f(x_0,x_1)x_3^2 + b(x_0,x_1)x_2 + e(x_0,x_1)x_3,
\end{equation}
where $a,d,f\in {\OO_\K}[x_0,x_1]$ are linear forms and
$b,e\in{\OO_\K}[x_0,x_1]$ are quadratic forms.  Moreover, $F= x_0 Q_0
- x_1 Q_1$ with quadratic forms $Q_0, Q_1 \in
{\OO_\K}[x_0,\ldots,x_3]$, and the nonsingularity of $X$ implies that
the morphism $\pi : X \to \PP^1_\K$ given on points by
\begin{equation*}
  (x_0:x_1:x_2:x_3)\mapsto
  \begin{cases}
    (x_0:x_1) &\text{ if }(x_0,x_1)\neq (0,0)\\
    (Q_1(\xx): Q_0(\xx)) &\text{ if }(Q_1(\xx), Q_0(\xx))\neq (0,0)
  \end{cases}
\end{equation*}
is well--defined. The fibre $\pi^{-1}(s:t)$ is the residual conic in
the plane $\Lambda_{(s:t)}$ defined by $tx_0-sx_1=0$. For any choice of
$(s,t)$, it is isomorphic to the plane conic $C_{(s,t)}$ defined by the quadratic form
\begin{equation}\label{eq:plane_conic}
Q_{(s,t)}:=a(s,t)x^2 + d(s,t)x z + f(s,t) z^2 + b(s,t)x y + e(s,t) y z=0
\end{equation}
via the isomorphism $\phi_{(s,t)} : \PP^2_K \to \Lambda_{(s:t)}$ given by
$(x:y:z)\mapsto (sy:ty:x:z)$. 

The discriminant locus of $\pi$
is given by
the quintic binary form
\[\Delta(s,t) := (ae^2- b d e +fb^2)(s,t),\]
which is separable owing to the nonsingularity of $X$ (see
\cite[II.6.4, Proposition 1]{MR1328833}). 
This confirms that the resultant
\[W_0:=\text{Res}(b(s,t),e(s,t))\]
must be in ${\OO_\K}\smallsetminus\{0\}$,
since the square of any common divisor of
$b(s,t)$ and $e(s,t)$ divides $\Delta(s,t)$.

Clearly, each $C_{(s,t)}$ contains the
rational point $(0:1:0)$, which is tantamount to
the
conic bundle morphism having a section
defined over $\K$. By a standard argument (see,
e.g., the paragraph following (1.6) in \cite{MR2838351}), we have
\begin{equation}\label{eq:pic_factors}
\rho_{X\!,\K}=2+r,
\end{equation}
where $r=r(X,\K)$ is the number of split singular fibres above closed
points of $\PP^1_\K$. Since the section meets
exactly one component of every singular fibre, we see that all
singular fibres are split. Consequently $r$ equals the number of
irreducible factors of $\Delta(s,t)$ in $\K[s,t]$.

\medskip

Using the conic fibration described above, we can reduce counting
points on $X$ to counting points on the fibres $\pi^{-1}(s:t)$ as follows:
\begin{equation*}
  N_{\K,H}(U,B) = \sum_{(s:t)\in\PP^1(\K)}N_{\K, H}(\pi^{-1}(s:t)\cap U,B).
\end{equation*}
Let $\mathcal{G}$ be a fundamental domain for the action of $\units$
on $(\K^\times)^2$ with the property that
\begin{equation}\label{eq:g_conjugates_same_size}
  \max\{\absv{s},\absv{t}\} \ll \max\{\abs{s}_w,\abs{t}_w\} \ll \max\{\absv{s},\absv{t}\}
\end{equation}
holds for all $v,w \in \archplaces$ and all $(s,t)\in \mathcal{G}$. We can
construct such a fundamental domain using, for example, the method from
\cite[Section 4]{MR2247898}.  Define the set
\begin{equation}\label{eq:def_B}
\mathcal{B}(x):=
\left\{(s,t)\in \Ok^2\cap \mathcal{G}
\ :
\begin{array}{l}
H((s:t))\leq x,
\\
s\Ok + t\Ok = \Ok
,\\
\pi^{-1}(s:t)\text{ is nonsingular}
\end{array}
\right\},
\end{equation}
where $H((s:t))$ is the usual exponential Weil height on $\PP^1(K)$.
For the purpose of acquiring a lower bound it is sufficient to
restrict the summation to points $(s:t)$ with representatives in
$\mathcal{B}(B^\delta)$, with the value $\delta :=
\beta/(2(1+\gamma))$. Then $N_{\K,H}(U,B)$ is larger than
\begin{equation*}
\sum_{
(s,t) \in \mathcal{B}(B^{\delta})}
N_{\K,H}(\pi^{-1}(s:t)\cap U, B)
=
\sum_{
(s,t) \in \mathcal{B}(B^{\delta})}
N_{\K,H}(\pi^{-1}(s:t), B) + O(B^{2\delta}),
\end{equation*}
by Schanuel's theorem, since every nonsingular conic contains at
most $54$ points lying on lines in $X$.

We use the isomorphism $\phi_{(s,t)}$ defined above to identify
$\pi^{-1}(s:t)$ with the plane conic $C_{(s,t)}$ given by
\eqref{eq:plane_conic}. The height $H$ on
$\pi^{-1}(s:t)$ is pulled back to the height $H\circ\phi_{(s,t)} = H_\vl$
on $C_{(s,t)}(\K)$, with $\lambda_v := \max\{\absv{s},\absv{t}\}$ for all
$v\in\archplaces$,
making the succeeding equality apparent,
\begin{equation*}
  N_{\K,H}(\pi^{-1}(s:t),B) = N_{\K,H_\vl}(C_{(s,t)},B).
\end{equation*}
Clearly, $\qfheight{Q_{(s,t)}} \ll H((s:t))^2$, and due to
\eqref{eq:g_conjugates_same_size} we have $M_\vl\ll 1$. Lemma
\ref{prop:conics} therefore
reveals that
\begin{equation*}
  N_{\K,H}(\pi^{-1}(s:t),B) = c(s,t) B + O(B^{1-\beta}H((s:t))^{2\gamma}),
\end{equation*}
with an explicit formula for $c(s,t):=c_{\K,\vl, C_{(s,t)}}$ given below Lemma \ref{prop:conics}.
Our choice of
$\delta$ implies that
\begin{equation}\label{eq:lower_bound_sum}
  N_{K,H}(U,B) \gg B
  \
\mathfrak{S}(B^{\delta})
+ O(B),
\end{equation}
where
\[\mathfrak{S}(x):=
\hspace{-0.3cm}
\sum_{
(s,t) \in \mathcal{B}(x)}
\hspace{-0.3cm}
c(s,t) .
\]
Our last undertaking 
is to show that
the quantity $\mathfrak{S}(B^\delta)$, the sum of the Peyre constants of the smooth conic fibres, 
provides the logarithmic factors appearing 
in Theorem \ref{thm:ssd}.

\section{The proof of Theorem \ref{thm:ssd}}\label{sec:proof}
For each place $v$ of $\K$, let $\sigma_v(s,t)$ be as in
\eqref{eq:sigma_v_arch},\eqref{eq:sigma_v_nonarch}, with the
parameterizing functions $\vq = \vq_{(s,t)}$ defined as in
\eqref{eq:def_q} for the quadratic form $Q_{(s,t)}$, and the norms
$\norm{\cdot}_{\vl,v}$ as in \eqref{eq:def_norms}, with $\lambda_v =
\max\{\absv{s},\absv{t}\}$. Let $\zeta_\K$ be the Dedekind zeta function of $\K$ and $\phi_\K$ be Euler's totient function for nonzero ideals of $\Ok$. Moreover, for nonzero ideals $\aaa$ of $\Ok$, we define the multiplicative function
\begin{equation*}
  \phi_\K^\dagger(\aaa) := \prod_{\ppp\mid\aaa}\left(1+\frac{1}{\N\ppp}\right),
\end{equation*}
where the product extends over all prime ideals $\ppp$ dividing $\aaa$. Clearly,
\begin{equation*}
  \frac{1}{\zeta_\K(2)} \leq \frac{\phi_\K^\dagger(\aaa)\phi_\K(\aaa)}{\N\aaa} \leq 1
\end{equation*}
holds for all $\aaa$.

\begin{lemma}[The non-archimedean densities]
\label{lem:nonarch_densities_est}
Let $\eta$ be any positive constant and
suppose $s,t\in\Ok$
fulfill $s\Ok + t\Ok   = \Ok$. Then we have
  \begin{equation*}
    \prod_{v\in\nonarchplaces}\sigma_v(s,t) \geq \frac{1}{\zeta_\K(2)}\sum_{\substack{\N\aaa \leq B^{\eta}\\\aaa\mid\Delta(s,t)\\\aaa + W_0\Ok = \Ok}}\left(\frac{\phi_\K(\aaa)}{\N\aaa}\right)^2.
  \end{equation*}
\end{lemma}

\begin{proof}
  Let $\rho_{(s,t)}^*(\aaa) := \rho_{\vq_{(s,t)}}^*(\aaa)$ as in
  \eqref{eq:def_rho}. Expanding the Euler product present in the lemma reveals its equality to
  \begin{equation*}
  \frac{1}{\zeta_\K(2)}\sum_{\aaa}\frac{\rho_{(s,t)}^*(\aaa)}{\phi_\K^\dagger(\aaa)\N\aaa}\geq \frac{1}{\zeta_\K(2)}\sum_{\substack{\N\aaa\leq B^\eta
\\\aaa\mid\Delta(s,t)\\\aaa+W_0\Ok=\Ok}}\frac{\rho_{(s,t)}^*(\aaa)}{\phi_\K^\dagger(\aaa)\N\aaa}.
  \end{equation*}
  Let $\aaa$ be an ideal of $\Ok$ with $\aaa \mid \Delta(s,t)$ and
  $\aaa + W_0\Ok = \Ok$.  We proceed to show that
  $\rho_{(s,t)}^*(\aaa) \geq \phi_\K(\aaa)$. Since $s^3W_0$ and
  $t^3W_0$ can be expressed as $\Ok$-linear combinations of $b(s,t)$
  and $e(s,t)$, we see that $b(s,t)\Ok + e(s,t)\Ok + \aaa = \Ok$. For
  every $\lambda \in \Ok/\aaa$ with $\lambda\Ok + \aaa = \Ok$, let $u
  := \lambda e(s,t)$ and $v := -\lambda b(s,t)$. Then
  $\vq_{(s,t)}(u,v) \equiv 0 \pmod \aaa$, and thus $\rho_{(s,t)}^*\geq
  \phi_\K(\aaa)$.
\end{proof}

\begin{lemma}[The archimedean densities]
\label{lem:arch_densities_est}
Suppose that $s$ and $t$ satisfy the assumption of Lemma \ref{lem:nonarch_densities_est}.
Then we have
  \begin{equation*}
  \prod_{v\in\archplaces}\sigma_v(s,t) \gg \frac{1}{H((s:t))^{2}}.
\end{equation*}
\end{lemma}

\begin{proof}
  The estimates
 \begin{align*}
 &\absv{b(s,t)}, \absv{e(s,t)} \ll
  \max\{\absv{s},\absv{t}\}^2 \ \ \ \ \text{and} \ \\
  &\absv{a(s,t)},\absv{d(s,t)},\absv{f(s,t)} \ll
  \max\{\absv{s},\absv{t}\}
\end{align*}
hold for each place $v\in\archplaces$. Hence, all $(y_1,y_2)\in \K_v^2$ satisfying
\[
\absv{y_1}, \absv{y_2}
\ll \max\{\absv{s},\absv{t}\}^{-1},
\]
with a suitably small absolute implied constant,
must
fulfill
$\norm{\vq_{(s,t)}(y_1,y_2)}_{\vl,v}\leq 1$.
We therefore get that
\begin{equation*}
  \prod_{v\in\archplaces}\sigma_v(s,t) \gg
 \prod_{v\in\archplaces}\max\{\absv{s},\absv{t}\}^{-2\locdegv} = H((s:t))^{-2}.\qedhere
\end{equation*}
\end{proof}
By \eqref{eq:constant_product_formula}, Lemma
\ref{lem:nonarch_densities_est} and Lemma
\ref{lem:arch_densities_est}, we obtain
\begin{equation}\label{eq:S_est_after_densities}
\mathfrak{S}(B^\delta)
\ \
 \gg
\hspace{-0,5cm}
\sum_{
(s,t) \in \mathcal{B}(B^{\delta})}
\frac{1}{H((s:t))^2}\sum_{\substack{\N\aaa \leq B^{\eta}\\\aaa\mid\Delta(s,t)\\\aaa + W_0\Ok = \Ok}}\left(\frac{\phi_\K(\aaa)}{\N\aaa}\right)^2.
\end{equation}
We observe that, apart from the condition $(s,t)\in \mathcal{G}$ from
\eqref{eq:def_B}, every expression involving $(s,t)$ in the above
formula is invariant under scalar multiplication of $(s,t)$ by units
in $\units$. Hence, we may replace $\mathcal{G}$ by another
fundamental domain $\mathcal{H}$, which will enable us to continue our
estimation of $\mathfrak{S}(x)$.  We obtain a fundamental domain
$\mathcal{H}_0$ for the action of $\units$ on $(\K\otimes_\QQ\RR)^\times$ by
making use of the embedding $ K^\times \to (K\otimes_\QQ \RR)^\times =
\prod_{v\in\archplaces}K_v^\times $ as well as the construction in
\cite[Section 4]{MR2247898} for the trivial distance functions
\[N_v : K_v \to [0,\infty) \ ,\  s\mapsto\absv{s}.\] The norm
$N:K\to\QQ$ extends to $K\otimes_\QQ\RR\to\RR$ in an obvious way. The
sets $\mathcal{H}_0(T):=\{s \in \mathcal{H}_0 \where \abs{N(s)}\leq
T\}$ clearly satisfy $\mathcal{H}_0(T) = T^{1/n}\mathcal{H}_0(1)$, and
by \cite[Lemma 3]{MR2247898}, the set $\mathcal{H}_0(1)$ is bounded
with Lipschitz-parameterizable boundary. This enables us to perform
lattice point counting arguments in the sets $\mathcal{H}_0(T)$ and
their translates, via \cite[Lemma 2]{MR2247898} for example.  We
choose $\mathcal{H} := (\mathcal{H}_0\cap K) \times \K^\times \subset (\K^\times)^2$
as our fundamental domain for the action of $\units$ on $\K^2$.

Partitioning into congruence classes modulo $\aaa$ yields
\begin{equation}\label{eq:S_est_congruence_classes}
\mathfrak{S}(B^\delta)
 \gg \sum_{\substack{\N\aaa \leq B^\eta\\\aaa+W_0\Ok = \Ok}}\left(\frac{\phi_\K(\aaa)}{\N\aaa}\right)^2\sum_{\substack{(\sigma, \tau)\bmod \aaa\\\sigma\Ok+\tau\Ok+\aaa=\Ok \\\aaa\mid\Delta(\sigma,\tau)}}
G_{\sigma,\tau}(B^\delta,\aaa)
,
\end{equation}
where
\[
G_{\sigma,\tau}(x,\aaa)
:=
\sum_{\substack{(s,t)\in (\Ok\cap \mathcal{H}_0)\times \Ok\\s\Ok+t\Ok=\Ok\\(s,t)\equiv (\sigma,\tau)\bmod \aaa\\H((s:t))\leq x
\\C_{(s,t)}\text{ nonsingular}}}\frac{1}{H((s:t))^2}.
\]
\begin{lemma}[Lattice point counting]
\label{lem:sum_crude_lower_bound}
  Let $\sigma\Ok + \tau\Ok + \aaa = \Ok$. Then
  \begin{equation*}
G_{\sigma,\tau}(x,\aaa)
\gg
\frac{\log x}{\N\aaa \ \phi_\K(\aaa) \ \phi_\K^\dagger(\aaa)} + O\left(x^{-\frac{1}{2n}}\log x \right).
  \end{equation*}
\end{lemma}
\begin{proof}
The discriminant $\Delta(s,t)$ is a quintic form
whence
the conic $C_{(s,t)}$ is singular for $(s,t)$ lying on one of at most $5$ lines through the origin in $\K^2$. 
Hence, there exists a constant $0 < \alpha < 1$, depending only on $F$ and $\K$, such that $C_{(s,t)}$ is nonsingular whenever $s,t\neq 0$ and $\absv{t} < \alpha\absv{s}$ holds for all $v \in \archplaces$. Observe that for such $(s,t)$ with $s\Ok+t\Ok = \Ok$ we have $H((s:t)) = \abs{N(s)}$. This shows that 
  \begin{equation*}
  G_{\sigma,\tau}(x,\aaa)
   \gg \sum_{\substack{(s,t)\in (\Ok\cap\mathcal{H}_0)\times
\Ok\\s\Ok+t\Ok=\Ok\\(s,t)\equiv (\sigma,\tau)\bmod \aaa\\ x^{1/2}\leq \abs{N(s)} \leq x \\0<\absv{t}<\alpha\absv{s} \forall v\in\archplaces}}
\hspace{-0,4cm}
\abs{N(s)}^{-2} =:
G(x), \ \text{say}.
  \end{equation*}
Using M\"obius inversion to remove the coprimality condition, we see that
\begin{equation*}
  G(x)
  =\sum_{\substack{\N\ddd\leq x
\\\ddd+\aaa=\Ok}}\mu_\K(\ddd)\sum_{\substack{s\in \ddd\cap \mathcal{H}_0\\s \equiv \sigma \bmod \aaa\\x^{1/2}\leq \abs{N(s)}\leq x}}
\abs{N(s)}^{-2}
\hspace{-0,4cm}
\sum_{\substack{t\in\ddd\\t\equiv \tau\bmod \aaa\\0<\absv{t}<\alpha\absv{s}\forall v\in\archplaces}}
\hspace{-0,4cm}
1 \ \ .
\end{equation*}
The condition $\ddd+\aaa=\Ok$ comes from our hypothesis that $\sigma\Ok + \tau\Ok + \aaa = \Ok$. The sum over $t$ is just counting ideal-lattice points in a translated ``box'', and their number is well known to be
\begin{equation*}
  \frac{c_\K\alpha^\dg \abs{N(s)}}{\N(\aaa\ddd)} + O\left(\left(\frac{\alpha^{\dg} \abs{N(s)}}{\N(\aaa\ddd)}\right)^{(\dg-1)/\dg} + 1 \right),
\end{equation*}
with a positive constant $c_\K$ depending only on $\K$ (see, for example, the proof of \cite[Lemma 7.1]{Marta}). Hence,
\begin{align*}
  G(x)
  &= \frac{c_\K\alpha^\dg}{\N\aaa}\sum_{\substack{\N\ddd\leq x \\\ddd+\aaa=\Ok}}\frac{\mu_K(\ddd)}{\N\ddd}\sum_{\substack{s\in \ddd\cap \mathcal{H}_0\\s \equiv \sigma \bmod \aaa\\ x^{1/2}
\leq \abs{N(s)}\leq x}}\frac{1}{\abs{N(s)}} + O\left(\sum_{\N\ddd\leq x}\sum_{\substack{s\in\ddd\cap \mathcal{H}_0\\x^{1/2}
\leq \abs{N(s)}\leq x}}\frac{1}{\abs{N(s)}^2}\right)\\
 & + O\left(\frac{1}{\N\aaa^{(\dg-1)/\dg}}\sum_{\N\ddd\leq x}\frac{1}{\N\ddd^{(\dg-1)/\dg}}\sum_{\substack{s\in \ddd\cap \mathcal{H}_0\\x^{1/2}\leq \abs{N(s)}\leq x}}\frac{1}{\abs{N(s)}^{1+1/\dg}}\right).
\end{align*}
The sums over $s$ in the error terms are taken over principal ideals of $\Ok$ contained in $\ddd$. For any $a > 0$, we have
\begin{equation*}
  \sum_{\substack{s\in \ddd\cap \mathcal{H}_0\\x^{1/2}\leq \abs{N(s)}\leq x}}\frac{1}{\abs{N(s)}^{1+a}} \ll \sum_{\substack{\bbb \in [\ddd^{-1}]\\\N\bbb \geq x^{1/2}\N\ddd^{-1}}}\frac{1}{\N(\bbb\ddd)^{1+a}} \ll \frac{1}{\N\ddd \  x^{a/2}}.
\end{equation*}
This shows that both error terms in the above expression for $G(x)$ are
$\ll x^{-1/(2n)} \log x$. Using the nice properties of our fundamental domain $\mathcal{H}_0$
and \cite[Lemma 2]{MR2247898}, we see that
\begin{equation*}
  \sharp\{s \in \ddd \cap \mathcal{H}_0 \where s\equiv \sigma\bmod \aaa\text{, }\abs{N(s)}\leq x\} = \frac{c_\K' x}{\N(\ddd\aaa)}+O\left(\left(\frac{x}{\N(\ddd\aaa)}\right)^{(n-1)/n}+1\right),
\end{equation*}
with a positive constant $c_\K'$ depending only on $\K$. Together with the Abel sum formula we are thus provided with
the asymptotic formula
\begin{equation*}
  \sum_{\substack{s\in \ddd\cap \mathcal{H}_0\\s \equiv \sigma \bmod \aaa\\ x^{1/2}\leq \abs{N(s)}\leq x }}\frac{1}{\abs{N(s)}} = \frac{c_\K'}{\N(\ddd\aaa)}\log(x) + O\left(x^{-1/(2n)}\right),
\end{equation*}
from which it is immediately apparent that
\begin{equation*}
  G(x)
   \gg \frac{\log x}{\N\aaa^2}\sum_{\substack{\N\ddd\leq x \\\ddd+\aaa=\Ok}}\frac{\mu_\K(\ddd)}{\N\ddd^2} + O\left(x^{-1/(2n)}\log x\right).
\end{equation*}
Finally, the obvious estimate
\begin{equation*}
  \sum_{\substack{\N\ddd\leq x\\\ddd+\aaa=\Ok}}\frac{\mu_\K(\ddd)}{\N\ddd^2} = \frac{\N\aaa}{\zeta_\K(2)\phi_\K(\aaa)\phi_\K^\dagger(\aaa)} + O\left(\frac{1}{x}\right)
\end{equation*}
allows us to complete the proof of the lemma.
\end{proof}
For any binary form $g\in \Ok[u,v]$, we define the multiplicative function 
$ \varrho_g^*(\aaa)$
on non--zero ideals of $\Ok$ by
\begin{equation*}
 \sharp\{(\sigma,\tau)\in(\Ok/\aaa)^2, \sigma\Ok+\tau\Ok+\aaa = \Ok\text{, }g(\sigma,\tau)\equiv 0\bmod \aaa\}
\end{equation*}
and
note that
its value is trivially bounded by $\N\aaa^2$. From the estimate \eqref{eq:S_est_congruence_classes} with $\eta := \delta/(7\dg)$ and Lemma \ref{lem:sum_crude_lower_bound}, we obtain
\begin{equation}
  \label{eq:S_est_ideals}
\mathfrak{S}(B^\delta)
\gg
  \hspace{-0.5cm}
\sum_{\substack{\N\aaa\leq B^{\delta/(7\dg)}\\a+W_0\Ok = \Ok}}\frac{\varrho_\Delta^*(\aaa)}{\N\aaa^2}\left(\frac{\phi_\K(\aaa)}{\N\aaa}\right)^2 + O(1).
\end{equation}
The following lemma is proved via an application of Wirsing's theorem
and its validity implies that of Theorem \ref{thm:ssd}.
\begin{lemma}
  For $x\geq 1$,
  \begin{equation*}
    \sum_{\substack{\N\aaa\leq x\\\aaa+W_0\Ok = \Ok}}\frac{\varrho_\Delta^*(\aaa)}{\N\aaa^2}\left(\frac{\phi_\K(\aaa)}{\N\aaa}\right)^2 \gg (\log x)^r.
  \end{equation*}
\end{lemma}

\begin{proof}
The form $\Delta$
factors as
  $a\Delta(s,t) = \prod_{i=1}^r
\Delta_i(s,t)$
  over $\K$
  for an appropriate value of $a=a(\K,F)\in \Ok$
   and irreducible forms
    $\Delta_i \in \Ok[s,t]$.
For $1\leq i \leq r$ with 
$\Delta_i(1,0)\neq 0$,
 let 
  $\delta_i(x) :=
  \Delta_i(x,1)\in\Ok[x].$
  We moreover define
for any polynomial $g\in\Ok[x]$ and any ideal $\aaa$ of $\Ok$,
\begin{equation*}
  \tau_g(\aaa) := \sharp\{s\in\Ok/\aaa\where g(s)\equiv 0\bmod \aaa\},
\end{equation*}
and we subsequently let
  \begin{equation*}
    \tau_i(\aaa):=
    \begin{cases}
      \tau_{\delta_i}(\aaa) &\text{ if }\Delta_i(1,0)\neq 0,\\
      \tau_{x}(\aaa)=1 &\text{ if }\Delta_i(1,0) = 0,
    \end{cases}
    \text{ and }a_i :=
    \begin{cases}
      \Delta_i(1,0) &\text{ if }\Delta_i(1,0)\neq 0,\\
      1 &\text{ if }\Delta_i(1,0)=0.
    \end{cases}
  \end{equation*}
The asymptotic relationships
  \begin{equation}\label{eq:tau_loglog}
    \sum_{\N\ppp\leq x}\frac{\tau_{i}(\ppp)}{\N\ppp} = \log\log x + O(1)\ \text{ and }  \sum_{\N\ppp\leq x}\frac{\tau_{i}(\ppp)\log(\N\ppp)}{\N\ppp} = \log x + O(1)
  \end{equation}
 follow from Landau's prime ideal theorem applied to
  $\K(\theta_i),$ where $\theta_i$ is a root of the irreducible
  polynomial $\delta_i.$ Since $\Delta$ is separable, all resultants
  $\Res(\Delta_i,\Delta_j)$
  are nonzero. Whence,
  upon introducing 
  \begin{equation*}
    W = W_F := a W_0 \prod_{i\neq j}
\Res(\Delta_i,\Delta_j)
\prod_{i=1}^ra_i \in \Ok\smallsetminus\{0\},
  \end{equation*}
  the equality
\[
    \varrho_\Delta^*(\ppp) = (\N\ppp-1)\sum_{i=1}^r\tau_{i}(\ppp)
 \]
  is rendered valid
  for each nonzero prime ideal $\ppp$
  of $\Ok$, coprime to $W$.
  This fact, along with \eqref{eq:tau_loglog},
  reveals that the multiplicative function defined by
  \begin{equation*}
    g(\aaa):=
    \begin{cases}
      \varrho_\Delta^*(\aaa)\phi_\K(\aaa)^2\N\aaa^{-4} &\text{ if }\aaa+W\Ok = \Ok \text{ and }\aaa\text{ squarefree,}\\
      0 &\text{ otherwise},
    \end{cases}
  \end{equation*}  
  satisfies the assumptions of Lemma \ref{thm:wirsing} with $k =
  r$.
We therefore get that there exists $c_g>0$
such that 
  \begin{equation*}
       \sum_{\substack{\N\aaa\leq x\\a+W\Ok = \Ok\\\aaa\text{ squarefree}}}\frac{\varrho_\Delta^*(\aaa)}{\N\aaa^2}\left(\frac{\phi_\K(\aaa)}{\N\aaa}\right)^2 = c_g(\log x)^r + O((\log x)^{r-1}),
  \end{equation*}
  an estimate
  which concludes our proof.
\end{proof}

\begin{acknowledgements}
  This collaboration was started during the AIM Workshop `\emph{Rational and
  integral points on higher-dimensional varieties}' in 2014. We would like to
  express our gratitude to the American Institute of Mathematics for
  its hospitality.

  We would like to thank Prof.~T.~Browning and Dr.~D.~Loughran for
  their helpful remarks.  While working on this paper the first author
  was supported by a Humboldt Research Fellowship for Postdoctoral
  Researchers and the second author was supported by the
  \texttt{EPSRC} grant \texttt{EP/H005188/1}.
\end{acknowledgements}

\bibliographystyle{alpha}
\bibliography{bibliography}
\end{document}